\documentclass[12pt, reqno]{amsart}
\usepackage{amsmath,amssymb}

\theoremstyle{plain}
\newtheorem{thm}{Theorem}[section]
\newtheorem*{main}{Main~Theorem}

\newtheorem{remark}{Remark}  

\newtheorem{corollary}[thm]{Corollary}
\newtheorem{lemma}[thm]{Lemma}

\renewcommand{\div}{\operatorname{div}}
\newcommand{\trace}{\operatorname{tr}}

\newcommand{\ip}[2]{\ensuremath{\langle #1 , #2 \rangle}}
\newcommand{\tp}{\texttt{p}}

\theoremstyle{definition}
\newtheorem{definition}{Definition}

\theoremstyle{remark}

\numberwithin{equation}{section}
\begin{document}
\title{The $\infty(x)$-equation in Riemannian Vector Fields}
\author{Thomas Bieske}
\address{Department of Mathematics \\
University of South Florida\\
Tampa, FL 33620, USA} 
\email{tbieske@mail.usf.edu}
\date{June 10, 2015}
\subjclass[2010]{35H20, 53C17,  49L25, 31B05, 31C12}
\keywords{Viscosity solutions, Riemannian vector fields, Infinite Laplacian}
\begin{abstract}
We employ Riemannian jets which are adapted to the Riemannian geometry to obtain the existence-uniqueness of viscosity solutions to the $\infty(x)$-Laplace equation in  Riemannian vector fields. Due to the differences between Euclidean jets and Riemannian jets, the Euclidean method of proof is not valid in this environment. 
\end{abstract}

\maketitle

\section{Introduction}  
Recently, the $\tp(x)$-Laplace equation and its limit equation, the $\infty(x)$-Laplace equation, have been the focus of much attention as a tool for exploring applications such as image restoration \cite{CLR} and electrorheological fluid flow \cite{R}. Linqvist and Luukari \cite{LL} recently proved existence-uniqueness of viscosity solutions to the $\infty(x)$-Laplace equation in (Euclidean) $\mathbb{R}^n$.  However, this proof is not valid in general Carnot- \\ Carath\'{e}odory spaces, such as Riemannian vector fields, because it relies on two important Euclidean properties, namely that the so-called viscosity penalty function is the square of the intrinsic distance and that the two first-order jet elements derived from the penalty function are equal. (These two phenomena are discussed more below.) The main result of this paper is that the lack of these phenomena in Riemannian vector fields can be overcome to produce existence-uniqueness of viscosity solutions in this environment. In particular, we prove the following theorem:
\begin{main}
    Let $\Omega$ be a bounded domain in $\mathbb{R}^n$ with Riemannian vector fields and let $f:\partial \Omega \to \mathbb{R}$ be a (Riemannian) Lipschitz function.
    Then the Dirichlet problem 
    \begin{eqnarray*}
\left\{ \begin{array}{cc}
-\Delta_{\mathfrak{X},\infty(x)} u  =  0  & \textmd{in}\  \Omega \\
u = f & \textmd{on}\ \partial \Omega
\end{array} \right.
\end{eqnarray*}
has a unique viscosity solution $u$. 
\end{main}

In Section 2, we review the main results and definitions from Riemannian vector fields. Section 3 is dedicated to existence-uniqueness of viscosity solutions and Section 4 details further properties of the viscosity solutions. 

\section{Riemannian Vector Fields}
\subsection{The Environment}
To create a Riemannian space, we begin with $\mathbb{R}^n$ and replace the Euclidean 
vector fields $\{\partial_{x_1}, \partial_{x_2},\ldots, \partial_{x_n}\}$
with an arbitrary collection of orthonormal vector fields or \textit{frame}
$$\mathfrak{X}=\{X_1,X_2,\ldots, X_n\}$$
consisting
of $n$ linearly independent smooth vector fields with the relation 
$$X_i(x)=\sum_{j=1}^n a_{ij}(x)\frac{\partial}{\partial x_j}$$
for some choice of smooth functions $a_{ij}(x)$.  Denote by $\mathbb{A}(x)$ the matrix whose
$(i,j)$-entry is $a_{ij}(x)$. We always assume that $\det(\mathbb{A}(x))\not=0$ in 
$\mathbb{R}^n$. 

The distance between points $x$ and $y$, denoted $d(x,y)$, is defined as the infimum of lengths of curves that join $x$ and $y$ with the additional requirement that the curves' tangent vectors lie in the span of the $X_i$'s. Using this distance, $\mathbb{R}^n$ with this frame is a metric space and,  unlike an arbitrary Carnot-Carath\'{e}odory space, this distance is locally comparable to Euclidean distance. We will discuss the importance of this fact below.

The natural gradient is the vector
$$D_{\mathfrak{X}}u=(X_1(u), X_2(u),\ldots, X_n(u))$$
and the natural second derivative is the $n\times n$ 
\textit{not necessarily symmetric} matrix  with entries $X_i(X_j(u))$. Because of the lack of symmetry, we introduce the symmetrized second-order derivative matrix with respect to this frame, given by $$(D^2_{\mathfrak{X}}u)^\star = \frac{1}{2}(D^2_{\mathfrak{X}}u + (D^2_{\mathfrak{X}}u)^t).$$

We can define function spaces $C^k$ and the Sobolev spaces $W^{1,p}$, etc with respect to this frame in the usual way. 

We may also define the $\infty$-Laplace operator
\begin{equation*}
\Delta_{\mathfrak{X},\infty} u  =
 \ip{(D^2_{\mathfrak{X}}u)^\star D_{\mathfrak{X}} u}{D_{\mathfrak{X}} u}.
\end{equation*}
This operator is the ``limit" operator of the $\tp$-Laplace operator (for $2< p<\infty$), which is given by 
\begin{eqnarray*}
\Delta_{\mathfrak{X},\tp} u & = &
 \|D_{\mathfrak{X}} u\|^{\tp-2}
\Delta_{\mathfrak{X}} u +(\tp-2) \|D_{\mathfrak{X}} u\|^{\tp-4}
 \Delta_{\mathfrak{X},\infty} u\\
& = &\div_{\mathfrak{X}} \ (\|D_{\mathfrak{X}} u\|^{\tp-2} D_{\mathfrak{X}} u)
\end{eqnarray*}
where the divergence is taken with respect to the frame $\mathfrak{X}$. 

Following \cite{LL}, we generalize these operators by replacing the constant $\tp$ with an appropriate function $\tp(x)\in C^1\cap W^{1,\infty}$ and scalar $k>1$ to obtain the $\tp(x)$-Laplace operator
\begin{eqnarray*}
\Delta_{\mathfrak{X},\tp(x)} u & = &
 \|D_{\mathfrak{X}} u\|^{k\tp(x)-2}
\Delta_{\mathfrak{X}} u +(k\tp(x)-2) \|D_{\mathfrak{X}} u\|^{k\tp(x)-4}
 \Delta_{\mathfrak{X},\infty} u\\
 & & \mbox{}+\|D_{\mathfrak{X}} u\|^{k\tp(x)-2}\ip{D_{\mathfrak{X}} u}{D_{\mathfrak{X}}k\tp(x)}\ln \|D_{\mathfrak{X}} u\| \\
& = &\div_{\mathfrak{X}} \ (\|D_{\mathfrak{X}} u\|^{k\tp(x)-2} D_{\mathfrak{X}} u).\end{eqnarray*}

The corresponding equation $\Delta_{\mathfrak{X},\tp(x)} u=0$ is the Euler-Lagrange equation associated to the energy functional $$\Bigg(\int_\Omega \frac{\|D_{\mathfrak{X}} u\|^{k\tp(x)}}{k\tp(x)}\;dx\Bigg)^\frac{1}{k}_.$$ Allowing $k\to\infty$, one has the tool for analysis of the extremal problem $$\min_u\max_x \|D_{\mathfrak{X}} u\|^{\tp(x)}.$$

Letting $k\to\infty$, we have $\Delta_{\mathfrak{X},\tp(x)} u \to \Delta_{\mathfrak{X},\infty(x)} u$ where 
\begin{equation*}
\Delta_{\mathfrak{X},\infty(x)} u =
\Delta_{\mathfrak{X},\infty} u + \|D_{\mathfrak{X}} u\|^{2}\ip{D_{\mathfrak{X}} u}{D_{\mathfrak{X}}\ln\tp(x)}\ln \|D_{\mathfrak{X}} u\|.
\end{equation*}

\subsection{Viscosity Solutions}
Because we will be considering viscosity solutions, we will recall the main definitions and properties. We begin with the Riemannian jets $J_{\mathfrak{X}}^{2,+}$ and $J_{\mathfrak{X}}^{2,-}$. 
(See \cite{BBM,BR} for a more complete analysis of such jets.) 
\begin{definition}
Let $u$ be an upper semi-continuous function. Consider the set
\begin{eqnarray*}
K_{\mathfrak{X}}^{2,+}u(x)  = \bigg\{
\varphi \in C^2\ \textmd{in a neighborhood of}\ x,  \varphi(x) = u(x),\\
\varphi(y) \geq u(y), \ y\neq x \ \text{in a neighborhood of}\ x\bigg\}.
\end{eqnarray*}
Each function $\varphi\in K_{\mathfrak{X}}^{2,+}u(x)$ determines a vector-matrix pair $(\eta,X)$ via the relations
\begin{equation*}\begin{array}{rcl}
\eta&  = & \big(X_1\varphi(x),X_2\varphi(x),\ldots, X_n\varphi(x)\big)\\
 X_{ij}&=& \frac{1}{2} \big(X_i(X_j(\varphi))(x)+X_j(X_i(\varphi))(x)\big).
\end{array}
\end{equation*}
We then define the \emph{second order superjet of $u$ at $x$} by $$J_{\mathfrak{X}}^{2,+}u(x)=\{(\eta,X):\varphi \in K^{2,+}u(x)\},$$
the \emph{second order subjet of $u$ at $x$} by $$J_{\mathfrak{X}}^{2,-}u(x)=-J_{\mathfrak{X}}^{2,+}(-u)(x)$$ and the set-theoretic closure 
\begin{eqnarray*}
\overline{J}_{\mathfrak{X}}^{2,+}u(x)=\{(\eta,X): & \exists \{x_n,\eta_n,X_n\}_{n\in \mathbb{N}}\ \textmd{with}\ (\eta_n,X_n)\in J_{\mathfrak{X}}^{2,+}u(x_n)\\ & \textmd{and}\ (x_n,u(x_n),\eta_n,X_n)\to(x,u(x),\eta,X)\}.
\end{eqnarray*}
\end{definition}

We then use these Riemannian jets to define viscosity $\infty(x)$-harmonic functions as follows:
\begin{definition}
A lower semi-continuous function  $v$  
  is \textbf{viscosity $\infty(x)$-superharmonic} in a bounded domain $\Omega$ if $v \not \equiv \infty$ in each component of $\Omega$ and for all $x_0 \in \Omega$,
whenever $(\xi, \mathcal{Y}) \in \overline{J}_{\mathfrak{X}}^{2,-} v(x_0)$, we have
$$-\Big(\ip{\mathcal{Y}\xi}{\xi} + \|\xi\|^{2}\ip{\xi}{D_{\mathfrak{X}}\ln\tp(x)}\ln \|\xi\| \Big)\geq 0.$$
An upper semi-continuous function  $u$  
  is  \textbf{viscosity $\infty(x)$-subharmonic} in a bounded domain $\Omega$ if $u \not \equiv -\infty$ in each component of $\Omega$ and for all $x_0 \in \Omega$,
whenever $(\eta,  \mathcal{X}) \in \overline{J}_{\mathfrak{X}}^{2,+} u(x_0)$, we have
$$-\Big(\ip{\mathcal{X}\eta}{\eta} + \|\eta\|^{2}\ip{\eta}{D_{\mathfrak{X}}\ln\tp(x)}\ln \|\eta\| \Big)\leq 0.$$
A function is \textbf{viscosity $\infty(x)$-harmonic} if it is both viscosity $\infty(x)$-subharmonic and viscosity $\infty(x)$-superharmonic.
\end{definition}

Similarly, we have the following definition concerning $\Delta_{\mathfrak{X},\tp(x)} u$.
\begin{definition}
A lower semi-continuous function  $v$  
  is \textbf{viscosity $\tp(x)$-superharmonic} in a bounded domain $\Omega$ if $v \not \equiv \infty$ in each component of $\Omega$ and for all $x_0 \in \Omega$,
whenever $(\xi, \mathcal{Y}) \in \overline{J}_{\mathfrak{X}}^{2,-} v(x_0)$, we have
$$-\Big(\|\xi\|^{k\tp(x)-2}\trace \mathcal{Y} +(k\tp(x)-2) \|\xi\|^{k\tp(x)-4}
 \ip{\mathcal{Y}\xi}{\xi}+\|\xi\|^{k\tp(x)-2}\ip{\xi}{D_{\mathfrak{X}}k\tp(x)}\ln \|\xi\| \Big)\geq 0.$$

An upper semi-continuous function  $u$ is  \textbf{viscosity $\tp(x)$-subharmonic} in a bounded domain $\Omega$ if $u \not \equiv -\infty$ in each component of $\Omega$ and for all $x_0 \in \Omega$,
whenever $(\eta,  \mathcal{X}) \in \overline{J}_{\mathfrak{X}}^{2,+} u(x_0)$, we have
$$-\Big(\|\eta\|^{k\tp(x)-2}\trace \mathcal{X} +(k\tp(x)-2) \|\eta\|^{k\tp(x)-4}
 \ip{\mathcal{X}\eta}{\eta}+\|\eta\|^{k\tp(x)-2}\ip{\eta}{D_{\mathfrak{X}}k\tp(x)}\ln \|\eta\|\Big)\leq 0.$$
A function is \textbf{viscosity $\tp(x)$-harmonic} if it is both viscosity $\tp(x)$-subharmonic and viscosity $\tp(x)$-superharmonic.
\end{definition}
\begin{remark}
In the above definitions, we may replace the right-hand side of each inequality by an arbitrary function. In that case, we use the term viscosity $\infty(x)$-subsolution, etc. 
\end{remark}
Our main tool is the Riemannian Maximum Principle \cite{BBM}, which we include for completeness. 
\begin{thm}\textbf{Riemannian Maximum Principle}\label{maxprin}
Let $u$ be  upper semicontinuous  in a bounded domain
$\Omega \subset \mathbb{R}^n$. Let $v$  be lower semicontinuous
in $\Omega$. Suppose that for    $x\in \partial \Omega$ we have
$$\limsup_{y\to x}u(y)\le \liminf_{y\to x}v(y),$$
where both sides are not $+\infty$ or $-\infty$ simultaneously.
If $u-v$ has a positive interior local maximum
$$
\sup_{\Omega} (u-v)  > 0$$ then we have: \\
For $\tau > 0$ we can find points $x_{\tau}, y_{\tau} \in\mathbb{R}^n$ 
such that
\begin{enumerate}
\item[i)]  $$\lim_{\tau\to\infty}\tau 
\psi(x_{\tau}, y_{\tau}) = 0,$$ where
$$\psi(x,y)=|x-y|^\alpha, $$for a fixed $\alpha \geq 2.$ (That is, $\psi$ is a power of the Euclidean distance.)
\item[ii)]  There exists a point $\hat  x \in\Omega$ such that $x_{\tau}\to \hat x$ (and so does $y_{\tau}$ by (i))
and 
$$\sup_{\Omega}(u-v) = u(\hat x) - v(\hat x) > 0,$$

\item[iii)] there exist symmetric matrices $\mathcal{X}_{\tau},
\mathcal{Y}_{\tau}$ and
vectors $\eta^+_{\tau} $, $\eta^-_\tau$ so that
\begin{enumerate}
\item[a)]$$(\eta^+_{\tau},\mathcal{X}_{\tau})\in \overline{J}_{\mathfrak{X}}^{2,+}u(x_{\tau}),$$
\item[b)]  $$(\eta^-_{\tau}, \mathcal{Y}_{\tau})  \in \overline{J}_{\mathfrak{X}}^{2,-}v(y_{\tau}),$$  
\item[c)] $$\eta^+_{\tau}- \eta^-_\tau=o(1)$$ and 
\item[d)]  $$ \mathcal{X}_{\tau} \leq \mathcal{Y}_{\tau} + o(1)$$ as
$\tau \to \infty$.  
\end{enumerate}
\end{enumerate}
\end{thm}
We will also use the following corollary. 
\begin{corollary}\cite[Corollary 2.1]{BR}\label{maxprincor}
Let $u$ and $v$ be as in Theorem \ref{maxprin}, and in addition, let one of $u$ or $v$ be locally Lipschitz. Let $\alpha=2$.  Then, for the vectors $\eta^+_{\tau}$ and $\eta^-_{\tau}$ and the matrices 
$\mathcal{X}_{\tau}$ and $\mathcal{Y}_{\tau}$
as in the theorem, we have 
\begin{eqnarray*}
\|\eta^+_{\tau}\|^2-\|\eta^-_{\tau}\|^2 & = & o(1) \\
\textmd{and \ \ }  \ip{\mathcal{X}_{\tau}\eta^+_{\tau}}{\eta^+_{\tau}}-
\ip{\mathcal{Y}_{\tau}\eta^-_{\tau}}{\eta^-_{\tau}} & = & o(1).
\end{eqnarray*}
\end{corollary} 
This corollary is a consequence of the facts that the Riemannian distance is comparable to the Euclidean distance and that the choice of penalty function $\psi(x,y)$ is the square of the Euclidean distance. Note, however, that even though the vectors $\eta^+_{\tau}$ and $\eta^-_{\tau}$ are not necessarily equal, we can still produce key estimates. 

\section{Existence-Uniqueness of $\infty(x)$-harmonic functions}
Let $\Omega$ be a bounded domain in $\mathbb{R}^n$ and $f:\partial \Omega\to\mathbb{R}$ be a (Riemannian) Lipschitz function. 

We will first establish the existence of $\infty(x)$-harmonic functions using Jensen's  auxiliary equations \cite{Je:ULE}: 
\begin{equation*}
\min \{\|D_{\mathfrak{X}} u\|^2 - \varepsilon, - \Delta_{\mathfrak{X},\infty(x)} u \}  =  0 
\ \ \textmd{and}\ \ \max \{\varepsilon -\|D_{\mathfrak{X}} u\|^2 , - \Delta_{\mathfrak{X},\infty(x)} u \} = 0
\end{equation*}
for a real parameter $\varepsilon>0$ . 
The procedure for existence of viscosity solutions to these equations (and viscosity $\infty(x)$-harmonic functions)  is identical to \cite[Section 4]{B:HG} and \cite[Section 2]{LL}, up to the obvious modifications. For completeness, we state the steps as one theorem and omit the proofs. 
\begin{thm}\cite{LL, B:HG}\label{exist}
We have the following results:
\begin{enumerate}
\item Let $\varepsilon\in\mathbb{R}$. If $u_k$ is a continuous potential-theoretic weak sub-(super-)solution with $u\in W^{1,k\tp(x)}(\Omega)$ to:
   \begin{eqnarray*}
\left\{ \begin{array}{cc}
-\Delta_{\mathfrak{X},k\tp(x)} u_k  =  \varepsilon^{k\tp(x)-1}  & \textmd{in}\  \Omega \\
u = f & \textmd{on}\ \partial \Omega
\end{array} \right.
\end{eqnarray*}
then it is a viscosity sub-(super-)solution.
\item Letting $k\to\infty$, we have $u_k\to u_\infty$ uniformly (possibly up to a subsequence) in $\Omega$ with $u_\infty\in W^{1,\infty}(\Omega)\cap C(\overline{\Omega})$.
\item The function $u_\infty$ is a viscosity solution to 
\begin{eqnarray*}
\min \{\|D_{\mathfrak{X}} u_\infty\|^2 - \varepsilon, - \Delta_{\mathfrak{X},\infty(x)} u_\infty \} =0 & \textmd{when} & \varepsilon>0 \\
\max \{\varepsilon -\|D_{\mathfrak{X}} u_\infty\|^2 , - \Delta_{\mathfrak{X},\infty(x)} u_\infty \}  =  0 & \textmd{when} & \varepsilon<0 \\
- \Delta_{\mathfrak{X},\infty(x)} u_\infty =0  & \textmd{when} & \varepsilon=0. 
\end{eqnarray*}  
\end{enumerate}
\end{thm}

In light of \cite[Lemma 5.6]{B:HG} and \cite[Lemma 2.2]{LL}, the Main Theorem follows from showing the uniqueness of viscosity solutions to the Jensen auxiliary equations. We will establish this result, and point out where we digress from the Euclidean proof. 

\begin{thm}\label{compinf}
Let $v=u_\infty$ be the viscosity solution from Theorem \ref{exist} to 
\begin{equation}\label{minequ}
\min \{\|D_{\mathfrak{X}} u\|^2 - \varepsilon, - \Delta_{\mathfrak{X},\infty(x)} u \} =0
\end{equation}
in a bounded domain $\Omega$. If $u$ is an upper semi-continuous viscosity subsolution to Equation \eqref{minequ} in $\Omega$ so that $u\leq v$ on $\partial \Omega$, then $u\leq v$ in $\Omega$. 
\end{thm}
\begin{proof}
Following \cite[Lemma 3.1]{LL} and \cite[Theorem 5.3]{B:HG}, we may assume WLOG that $v$ is a strict viscosity supersolution. Suppose $$\sup_{\Omega} (u-v)>0$$ and let $\psi(x,y)=|x-y|^2$ be the square of the Euclidean distance between the points $x$ and $y$. By the Riemannian Maximum Principle (Theorem \ref{maxprin}), there are points $x_\tau$ and $y_\tau$ in $\Omega$ (for sufficiently large $\tau$) with the property that there are vectors $\eta^+_\tau, \eta^-_\tau$ and symmetric matrices $\mathcal{X}_\tau, \mathcal{Y}_\tau$ so that 
$$(\eta^+_\tau, \mathcal{X}_\tau)\in \overline{J}^{2,+}u(x_\tau)\ \ \textmd{and}\ \ 
(\eta^-_\tau, \mathcal{Y}_\tau)\in \overline{J}^{2,-}v(y_\tau).$$

Since $u$ is a viscosity subsolution and $v$ a strict viscosity supersolution, we have, for some $\mu>0$, 
\begin{eqnarray*}
0 & \geq & \min \{\|\eta^+_\tau\|^2 - \varepsilon, 
-\ip{\mathcal{X}_\tau\eta^+_\tau}{\eta^+_\tau} - \|\eta^+_\tau\|^{2}\ip{\eta^+_\tau}{D_{\mathfrak{X}}\ln\tp(x_\tau)}\ln \|\eta^+_\tau\|\}\\
0<\mu & \leq &  \min \{\|\eta^-_\tau\|^2 - \varepsilon, 
-\ip{\mathcal{Y}_\tau\eta^-_\tau}{\eta^-_\tau} - \|\eta^-_\tau\|^{2}\ip{\eta^-_\tau}{D_{\mathfrak{X}}\ln\tp(y_\tau)}\ln \|\eta^-_\tau\|\}.
\end{eqnarray*}

Subtracting these equations, we obtain
\begin{eqnarray}
0<\mu & \leq & \max \{\|\eta^-_\tau\|^2 - \|\eta^+_\tau\|^2, 
\ip{\mathcal{X}_\tau\eta^+_\tau}{\eta^+_\tau}-\ip{\mathcal{Y}_\tau\eta^-_\tau}
{\eta^-_\tau} \nonumber \\
 & &\hspace{.5in} \mbox{}+\|\eta^+_\tau\|^{2}\ip{\eta^+_\tau}{D_{\mathfrak{X}}\ln\tp(x_\tau)}\ln \|\eta^+_\tau\| \label{contr} \\ 
 & & \hspace{.5in}\mbox{}- \|\eta^-_\tau\|^{2}\ip{\eta^-_\tau}{D_{\mathfrak{X}}\ln\tp(y_\tau)}\ln \|\eta^-_\tau\|\}.\nonumber
\end{eqnarray}

Here is where the proof diverges from the Euclidean case. In the Euclidean case, the vectors $\eta^+_\tau$ and $\eta^-_\tau$ are equal, rapidly leading to a contradiction in Equation \eqref{contr} as $\tau\to\infty$. However, in the Riemannian environment, these vectors are not, in general, equal. So, we will have to estimate the right-hand side more carefully. 

Since $v$ is locally Lipschitz, can invoke Corollary \ref{maxprincor} to obtain 
$$\|\eta_{y_\tau}\|^2 - \|\eta_{x_\tau}\|^2 \to 0\ \ \textmd{and}\ \ 
\ip{\mathcal{X}_\tau\eta^+_\tau}{\eta^+_\tau}-\ip{\mathcal{Y}_\tau\eta^-_\tau}
{\eta^-_\tau}\to 0$$ as  $\tau\to\infty$. 

Thus, we are left to consider $$\|\eta^+_\tau\|^{2}\ip{\eta^+_\tau}{D_{\mathfrak{X}}\ln\tp(x_\tau)}\ln \|\eta^+_\tau\|- \|\eta^-_\tau\|^{2}\ip{\eta^-_\tau}{D_{\mathfrak{X}}\ln\tp(y_\tau)}\ln \|\eta^-_\tau\|.$$
We begin by expressing the sum as 
\begin{eqnarray*}
\lefteqn{\|\eta^+_\tau\|^{2}\ip{\eta^+_\tau}{D_{\mathfrak{X}}\ln\tp(x_\tau)}\ln \|\eta^+_\tau\|- \|\eta^-_\tau\|^{2}\ip{\eta^-_\tau}{D_{\mathfrak{X}}\ln\tp(y_\tau)}\ln \|\eta^-_\tau\|}&&\\  
& = & \|\eta^+_\tau\|^{2}\ip{\eta^+_\tau}{D_{\mathfrak{X}}\ln\tp(x_\tau)}\bigg(\ln \|\eta^+_\tau\|-\ln \|\eta^-_\tau\|\bigg) \\
& + &\bigg(\|\eta^+_\tau\|^{2}- \|\eta^-_\tau\|^{2}\bigg)\ip{\eta^+_\tau}{D_{\mathfrak{X}}\ln\tp(x_\tau)}\ln \|\eta^-_\tau\| \\
 & + & \|\eta^-_\tau\|^{2}\bigg(\ip{\eta^+_\tau}{D_{\mathfrak{X}}\ln\tp(x_\tau)}-
 \ip{\eta^-_\tau}{D_{\mathfrak{X}}\ln\tp(x_\tau)}\bigg)\ln \|\eta^-_\tau\|\\
 & + & \|\eta^-_\tau\|^{2}\bigg(\ip{\eta^-_\tau}{D_{\mathfrak{X}}\ln\tp(x_\tau)}-
\ip{\eta^-_\tau}{D_{\mathfrak{X}}\ln\tp(y_\tau)}\bigg)\ln \|\eta^-_\tau\| \\
  & = & I + II + III + IV.
\end{eqnarray*}
To estimate each term, we will need the following Lemma.
\begin{lemma}\label{help}
For some constant $K>0$, the vector $\eta^-_\tau$ satisfies $$\sqrt{\varepsilon} < \|\eta^-_\tau\| < K.$$  In addition, for that same $K$, and sufficiently large $\tau$, the vector $\eta^+_\tau$ satisfies 
$$\frac{1}{2}\sqrt{\varepsilon}<\|\eta^+_\tau\|<K.$$ 
As a consequence, there is a constant $L$ so that $$\Big|\ln \|\eta^-_\tau\|\Big|<L<\infty.$$ 
\end{lemma}
\begin{proof}
Since $v$ is a strict supersolution to Equation \eqref{minequ}, we have $0<\varepsilon<
\|\eta^-_\tau\|^2$.  Next, since $v$ is locally Lipschitz, the proof of Corollary \ref{maxprincor} shows there is a constant $C$  so that $$\tau \psi(x_\tau, y_\tau)^{\frac{1}{2}} < C.$$ By the proof of the Riemannian Maximum Principle (Theorem \ref{maxprin}), $$\eta^-_\tau=-\tau \mathbb{A}(y_\tau)D_y(\psi(x_\tau,y_\tau))$$ where $D_y$ denotes Euclidean differentiation in the $y$-variable and $\mathbb{A}(y_\tau)$ is the coefficient matrix of the frame at $y_\tau$ in terms of the standard Euclidean vectors (Section 2.1). 
Because $\psi(x_\tau,y_\tau)=|x_\tau-y_\tau|^2$ and $y_\tau \in \Omega$, a bounded domain, we conclude that for some finite constant $C_1$,
$$\|\eta^-_\tau\| \leq C_1.$$ Similarly, 
$$\eta^+_\tau=\tau \mathbb{A}(x_\tau)D_x(\psi(x_\tau,y_\tau))$$ so that 
for some finite constant $C_2$, $$\|\eta^+_\tau\| \leq C_2.$$ By part iii(c) of the Riemannian Maximum Principle (Theorem \ref{maxprin}), for sufficiently large $\tau$, 
$0<\frac{1}{2}\sqrt{\varepsilon}<\|\eta^+_\tau\|$.
The Lemma then follows. 
\end{proof}

\noindent\textbf{Term I:} The absolute value of the first term is controlled by 
$$\|\eta^+_\tau\|^{3}\|D_{\mathfrak{X}}\ln\tp(x_\tau)\|_{L^{\infty}}\bigg|\ln \|\eta^+_\tau\|-\ln \|\eta^-_\tau\|\bigg|.$$ Using Lemma \ref{help}  and the fact that $1<\tp(x)\in C^1(\Omega)\cap W^{1,\infty}(\Omega)$, we have that this is, in turn, controlled by 
 $$C\bigg(\ln \|\eta^+_\tau\|-\ln \|\eta^-_\tau\|\bigg)$$ for some finite constant $C$. Suppose that \begin{equation*}
\bigg|\ln \|\eta^+_\tau\|-\ln \|\eta^-_\tau\|\bigg|  =  \Big|\ln \bigg( \frac{ \|\eta^+_\tau\|}{ \|\eta^-_\tau\|}\bigg) \bigg|=\ln \bigg( \frac{ \|\eta^+_\tau\|}{ \|\eta^-_\tau\|}\bigg).
 \end{equation*}
Then 
\begin{equation*}
\bigg|\ln \|\eta^+_\tau\|-\ln \|\eta^-_\tau\|\bigg| \leq \ln \bigg(1+\frac{\|\eta^+_\tau-\eta^-_\tau\|}{\|\eta^-_\tau\|}\bigg).
 \end{equation*}
 Lemma \ref{help} and the Riemannian Maximum Principle then imply as $\tau\to\infty$, $$\bigg|\ln \|\eta^+_\tau\|-\ln \|\eta^-_\tau\|\bigg|\to 0.$$
 If instead, 
\begin{equation*}
\bigg|\ln \|\eta^+_\tau\|-\ln \|\eta^-_\tau\|\bigg|  =  \Big|\ln \bigg( \frac{ \|\eta^+_\tau\|}{ \|\eta^-_\tau\|}\bigg) \bigg|=-\ln \bigg( \frac{ \|\eta^+_\tau\|}{ \|\eta^-_\tau\|}\bigg)
=\ln \bigg( \frac{ \|\eta^-_\tau\|}{ \|\eta^+_\tau\|}\bigg)
 \end{equation*}
 then a symmetric argument gives $$\bigg|\ln \|\eta^+_\tau\|-\ln \|\eta^-_\tau\|\bigg|\to 0$$ as $\tau\to\infty$, so that 
 $$\|\eta^+_\tau\|^{2}\ip{\eta^+_\tau}{D_{\mathfrak{X}}\ln\tp(x_\tau)}\bigg(\ln \|\eta^+_\tau\|-\ln \|\eta^-_\tau\|\bigg)\to 0$$ as $\tau\to\infty$.  
 
\noindent\textbf{Term II:} By Corollary \ref{maxprincor}, $$\|\eta^+_\tau\|^{2}- \|\eta^-_\tau\|^{2}\to 0 \ \ \textmd{as}\ \  \tau\to\infty.$$  Using Lemma \ref{help}, we have 
$$\bigg|\ip{\eta^+_\tau}{D_{\mathfrak{X}}\ln\tp(x_\tau)}\ln \|\eta^-_\tau\|\bigg|
\leq C \|D_{\mathfrak{X}}\ln\tp(x_\tau)\|_{L^{\infty}}.$$ As in Term I, this term is bounded and so as $\tau\to\infty$
$$\bigg(\|\eta^+_\tau\|^{2}- \|\eta^-_\tau\|^{2}\bigg)\ip{\eta^+_\tau}{D_{\mathfrak{X}}\ln\tp(x_\tau)}\ln \|\eta^-_\tau\|\to 0.$$  

\noindent\textbf{Term III:} As in the previous terms, the absolute value of this term is controlled by $$C\|\eta^+_\tau-\eta^-_\tau\|.$$ By the Riemannian Maximum Principle, we then have as $\tau\to\infty$,
$$ \|\eta^-_\tau\|^{2}\bigg(\ip{\eta^+_\tau}{D_{\mathfrak{X}}\ln\tp(x_\tau)}-
 \ip{\eta^-_\tau}{D_{\mathfrak{X}}\ln\tp(x_\tau)}\bigg)\ln \|\eta^-_\tau\|\to 0.$$
 
\noindent\textbf{Term IV:}  By the Riemannian Maximum Principle, both $x_\tau$ and $y_\tau$ converge to a point $\hat{x}$. By the regularity of $\tp(x)$, we have the absolute value of Term IV is controlled by $$C\bigg(D_{\mathfrak{X}}\ln\tp(x_\tau)-D_{\mathfrak{X}}\ln\tp(y_\tau)\bigg)\to 0$$ as $\tau\to\infty$. 
\end{proof}
An analogous argument produces the following Corollary.
\begin{corollary}\label{compinfcor}
Let $v=u_\infty$ be the viscosity solution from Theorem \ref{exist} to 
\begin{equation}\label{maxequ}
\max \{\varepsilon-\|D_{\mathfrak{X}} u\|^2, - \Delta_{\mathfrak{X},\infty(x)} u \} =0
\end{equation}
in a bounded domain $\Omega$. If $u$ is an lower semi-continuous viscosity supersolution to Equation \eqref{maxequ} in $\Omega$ so that $u\geq v$ on $\partial \Omega$, then $u\geq v$ in $\Omega$. 
\end{corollary}

\section{A Harnack Inequality}
We include a Harnack inequality for completeness. First, we have the following lemma whose proof is identical to \cite[Lemma 4.1]{LL} and omitted. 
\begin{lemma}
Let $u$ be a positive viscosity $\infty(x)$-harmonic function and $\zeta$ a positive, compactly supported smooth function. Then 
$$\sup_{x\in\Omega}\bigg|D_{\mathfrak{X}}\zeta(x)D_{\mathfrak{X}}\ln u(x)\bigg|^{\tp(x)} \leq \sup_{x_\in\Omega}\bigg|D_{\mathfrak{X}}\zeta(x)+\zeta(x)\ln\bigg(\frac{\zeta(x)}{u(x)}\bigg)D_{\mathfrak{X}}\ln \tp(x)\bigg|^{\tp(x)}.$$
\end{lemma}
As in \cite[Section 4]{LL}, we have the following Harnack inequality as a consequence. 
\begin{thm}
Let $u$ be a positive viscosity $\infty(x)$-harmonic function. Let $B_r$ be a ball of radius $r>0$ contained in the bounded domain $\Omega$. Let $B_{2r}$ be the concentric ball of twice the radius also contained in $\Omega$. Then 
$$\sup_{x\in B_r} u(x) \leq C \big(\inf_{x\in B_r} u(x)+r\big)$$ for some constant $C$ 	depending on $\sup_{x\in B_{2r}}u(x)$. 
\end{thm}

\end{document}